\documentclass[10pt]{article}
\usepackage[left=1in,top=1in,right=1in,bottom=1in,letterpaper]{geometry}

\usepackage[usenames]{color}

\usepackage[color,final]{showkeys}

\usepackage{amsmath,amssymb,amsthm}
\usepackage{latexsym,amsfonts,amscd,amsxtra,amstext}
\usepackage{algorithm,algorithmic}

\usepackage{epsfig} 
\usepackage{threeparttable}
\usepackage{index}
\usepackage{subfigure}

\usepackage[normalem]{ulem} 

\newcommand{\cC}{{\mathcal{C}}}

\newcommand{\cF}{{\mathcal{F}}}

\newcommand{\cL}{{\mathcal{L}}}

\newcommand{\cR}{{\mathcal{R}}}
\newcommand{\cS}{{\mathcal{S}}}

\newcommand{\cX}{{\mathcal{X}}}

\newcommand{\RR}{\mathbb{R}}

\newcommand{\sign}{\mathrm{sign}}

\newcommand{\dom}{{\mathrm{dom}}} 
\newcommand{\Proj}{{\mathrm{Proj}}}

\newcommand{\prj}{{\mathrm{prj}}}

\DeclareMathOperator{\shrink}{shrink} 

\newcommand{\bc}{\begin{center}}
\newcommand{\ec}{\end{center}}

\newcommand{\bdm}{\begin{displaymath}}
\newcommand{\edm}{\end{displaymath}}

\newcommand{\beq}{\begin{equation}}
\newcommand{\eeq}{\end{equation}}

\newcommand{\bfl}{\begin{flushleft}}
\newcommand{\efl}{\end{flushleft}}

\newcommand{\bt}{\begin{tabbing}}
\newcommand{\et}{\end{tabbing}}

\newcommand{\beqn}{\begin{eqnarray}}
\newcommand{\eeqn}{\end{eqnarray}}

\newcommand{\beqs}{\begin{align*}} 
\newcommand{\eeqs}{\end{align*}}  

\newtheorem{theorem}{Theorem}

\newtheorem{definition}{Definition}

\newtheorem{remark}{Remark}
\newtheorem{lemma}{Lemma}

\newtheorem{example}[remark]{Example}

\begin{document}

\title{Gradient methods for convex minimization: better rates under weaker conditions}

\author{Hui Zhang\thanks{Department of Mathematics and Systems Science,
College of Science, National University of Defense Technology,
Changsha, Hunan,  China. Email: {hhuuii.zhang@gmail.com}}
\and Wotao Yin\thanks{Department of Computational and Applied Mathematics, Rice University, Houston, Texas, US. Email: {wotao.yin@rice.edu} }
}
\date{\today}

\maketitle

\begin{abstract}
The convergence behavior of gradient methods for minimizing convex differentiable functions is one of the core questions in convex optimization. This paper shows that their well-known complexities can be achieved under conditions weaker than the commonly accepted ones. We relax the common gradient Lipschitz-continuity condition and strong convexity condition to ones that hold only over certain line segments. Specifically, we establish complexities $O(\frac{R}{\epsilon})$ and $O(\sqrt{\frac{R}{\epsilon}})$ for the ordinary and accelerated gradient methods, respectively, assuming that $\nabla f$ is Lipschitz continuous with constant $R$ over the line segment joining $x$ and $x-\frac{1}{R}\nabla f$ for each $x\in\dom f$. Then we improve them to $O(\frac{R}{\nu}\log(\frac{1}{\epsilon}))$ and $O(\sqrt{\frac{R}{\nu}}\log(\frac{1}{\epsilon}))$ for function $f$ that also satisfies the secant inequality $\langle \nabla f(x), x- x^*\rangle\ge \nu\|x-x^*\|^2$ for each $x\in \dom f$ and its projection $x^*$ to the minimizer set of $f$. The secant condition is also shown to be necessary for the geometric decay of solution error. Not only are the relaxed conditions met by more functions, the restrictions give smaller $R$ and larger $\nu$ than they are without the restrictions and thus lead to better complexity bounds. We apply these results to sparse optimization and demonstrate a faster algorithm.

\end{abstract}


\section{Introduction}
Owing  much to the fast development in
signal/image processing, compressive sensing, statistical and machine learning, and parallel computing, we have witnessed the  (revived) popularity of gradient methods, which are easy to program, have relatively low per-iteration complexities, and are often among the best options for obtaining moderately accurate solutions for large-scale optimization problems.

This paper considers the convex unconstrained optimization problem: \begin{equation}\label{opt1}
 f^*:=\min_{x\in \RR^n} f(x)
\end{equation}
where $f: \RR^n\rightarrow \RR$ is a differentiable convex function.  We assume throughout the paper that the set of optimal solutions $\mathcal{X}^*$ is nonempty and closed and thus  $f^*\in \RR$ is attainable.  For simplicity, we assume $\dom f=\RR^n$.
Most of the discussions in this paper hold if we impose $x\in\dom f$ rather than $x\in \RR^n$.

The gradient descent iteration is
\beq\label{graddes}x^{(k+1)}= x^{(k)}-\tau \nabla f(x^{(k)}).
\eeq
Its convergence rates have been established for two major classes of functions  \cite{n1,n2,n21}: The first class, denoted by $\mathcal{F}_L(\RR^n)$, consists of the convex functions with Lipschitz continuous gradients, namely,
\begin{align}
\nonumber
f\in \mathcal{F}_L(\RR^n)~\Longleftrightarrow~& f~\mbox{is differentiable~and}\\
&\|\nabla f(x)-\nabla f(y)\|\leq L\|x-y\|, \quad\forall x, y\in\RR^n,\label{Lip}
\end{align}
where $L>0$ is the  Lipschitz constant of $\nabla f$; the second class, denoted by $\mathcal{S}_{L,\mu}(\RR^n)$, is a subclass of $\cF_L(\RR^n)$ in which the functions are also strongly convex, namely,
\begin{align}
\nonumber
f\in \mathcal{S}_{\mu,L}(\RR^n)~\Longleftrightarrow~& f\in \mathcal{F}_L(\RR^n)~\mbox{and}\\
&\langle \nabla f(x)-\nabla f(y), x-y\rangle \geq \mu \|x-y\|^2, \quad\forall x, y\in\RR^n,\label{SC}
\end{align}
where $\mu>0$ is  the convex modulus  of $f$. Geometrically, if $f\in \mathcal{F}_L$,  $\nabla f$ cannot change too quickly; the curvature of $f$ (assuming $f\in C^2$) is upper bounded by $L$. If $f\in S_{\mu,L}$,   $\nabla f$ cannot change too slowly either; the curvature of $f$ (assuming $f\in C^2$) is lower bounded by $\mu$. One might be more familiar certain equivalent conditions of  \eqref{Lip} and \eqref{SC}.

\begin{table}[ht]
\begin{center}\begin{tabular}{c|c|c|c}\hline
function & 1st-order oracle & ordinary gradient & accelerated gradient\\
class & lower bound & method & method\\\hline
~&~ &~ & ~\\
$\cF_L(\RR^n)$ & $O\left(\sqrt\frac{L}{{\epsilon}}\right)$ & $O\left(\frac{L}{\epsilon}\right)$ &  $O\left(\sqrt\frac{L}{{\epsilon}}\right)$ \\[15pt]
$\cS_{L,\mu}(\RR^n)$ & $O\left(\sqrt{\frac{L}{\mu}}\log\frac{1}{\epsilon}\right)$ & $O\left(\frac{L}{\mu}\log\frac{1}{\epsilon}\right)$ &  $O\left(\sqrt{\frac{L}{\mu}}\log\frac{1}{\epsilon}\right)$\\[10pt]\hline
\end{tabular}\end{center}
\caption{Complexities of minimizing a convex differentiable function to $\epsilon$-accuracy}\label{complexity}
\end{table}

For any $f\in \cF_L$, iteration \eqref{graddes} reduces $f^k=f(x^{(k)})$ at the rate of $O(\frac{L}{k})$; hence, it takes $O(\frac{L}{\epsilon})$ iterations to guarantee $f^k< f^*+\epsilon$. For any $f\in S_{\mu,L}$, the rate is improved to $O(\frac{L-\mu}{L+\mu})^{2k}$. Therefore, it only takes $O(\frac{L}{\mu}\log(\frac{1}{\epsilon}))$ iterations.

In the seminal paper \cite{n1}, Nesterov presents an accelerated gradient descent iteration. For functions in $\cF_L$, its complexity is $O(\sqrt{\frac{L}{\epsilon}})$.  In papers \cite{n2,n3}, he generalizes the method to more  function classes. In particular, if $f\in S_{\mu,L}$, the complexity is $O(\sqrt{\frac{L}{\mu}}\log(\frac{1}{\epsilon}))$. He  gives examples of functions on which no gradient-based methods can perform fundamentally better. So, his method has  the optimal worst-case complexities; for more detail, see  book \cite{n21}.
The complexities discussed above are summarized in Table \ref{complexity}.

\subsection{Contributions}
We show that global Lipschitz continuity of $\nabla f$ is not  necessary for deriving the sublinear bounds in Table \ref{complexity}. If $\nabla f$ is  Lipschitz continuous with constant $R>0$ restricted to the line segments joining $x$ and $x-(1/R)\nabla f(x)$, for $x=x^{(0)},x^{(1)},\ldots$, or simply $x\in\RR^n$, then the ordinary and accelerated gradient descent methods have complexities $O(R/\epsilon)$ and $O(\sqrt{R/\epsilon})$, respectively. We believe that some researchers, especially those who study line search methods, might be aware of this result though we do not find it in the literature. Our analysis in fact hints a backtracking line search method that achieves the same complexities without the knowledge of $R$. It is worth noting that the recent paper \cite{SGB} presents a skillful line search method that improves the Nesterov's accelerated gradient method.

On the other hand, the Lipschitz continuity of $\nabla f$ alone gives at best  the rather weak $O(1/\epsilon)$ and $O(1/\sqrt{\epsilon})$ complexities.  It is commonly know that the strong convexity of $f$ enables the much better complexity of $O(\log(1/\epsilon))$. However, most convex functions are not strongly convex. Hence, it is interesting to relax the conditions and still establish a linear convergence rate. We show that an inequality resembling \eqref{SC} but  concerning just the secant between $x$ and its projection to $\cX^*$ is ultimately responsible for linear convergence. The inequality imposes a positive lower bound on the \emph{average curvature} between $x$ and the solution set and is shown to be both sufficient and necessary for the geometric decay of solution error.

\subsection{Outline of the paper}

The rest of the paper is organized as follows. Section \ref{weakconds} defines new properties of functions along with examples and discussions. Section \ref{sc:main} describes the convergence and complexity results. Section \ref{sc:augl1} applies these results to the augmented $\ell_1$ model and presents numerical results of sparse signal recovery. Finally, Section \ref{sc:concl} concludes this paper.

\section{Weakened conditions}\label{weakconds}
For any two vector $u, v\in \RR^n$, we let the set of points on the line segment between $u$ and $v$ be denoted by $\lfloor u, v\rfloor$, i.e., $$\lfloor u, v\rfloor=\{w\in \RR^n: w=\lambda u+(1-\lambda)v, 0\leq \lambda \leq 1\}.$$
\begin{definition}[Restricted Lipschitz-continuous gradient -- RLG($R$)]\label{rlg}
A function $f(x):\RR^n\rightarrow \RR$ has a restricted Lipschitz-continuous gradient (RLG) with constant $R\ge 0$
if it is differentiable and obeys
\beq\label{eq:rlg}
\|\nabla f(x)-\nabla f(y)\|\le R \|x-y\|,\quad \forall (x,y)\in\Omega,
\eeq
where
\beq\label{Omeg}\Omega = \bigcup_{z\in \RR^n}\{(x,y):x,y\in \lfloor z,z-(1/R)\nabla f(z)\rfloor\}.
\eeq
\end{definition}
This definition requires $\nabla f$ not to change too quickly over the specified downhill line segments \eqref{Omeg}. Constant $R$ can generally be smaller than the global Lipschitz constant $L$.

\begin{definition}[Restricted secant inequality -- RSI($\nu$)] \label{rsi}
A function $f(x):\RR^n\rightarrow \RR$ satisfies the restricted secant inequality (RSI)  with constant $\nu>0$ if it is differentiable and obeys
\begin{equation}\label{eq:rsi}
\langle \nabla  f(x)- \nabla f(x_{\prj}), x-{x_{\prj}}\rangle \geq \nu \|x-{x_{\prj}}\|^2,
\end{equation}
where ${x_{\prj}}=\Proj_{\cX^*}(x)$ is the projection of $x$ onto the solution set $\mathcal{X}^*$. Such $f$ is called an RSI function.
\end{definition}
Note that $\nabla f(x_{\prj})=0$ by definition. Constant $\nu$ can be viewed as a lower bound of the average curvature of $f$ between $x$ and $x_\prj$.  Since the goal of minimization is to reach the solution set $\mathcal{X}^*$, in order to have linear convergence, it turns out only the ``average minimum curvature'' between the current  $x$ and its projection $x_\prj$  matters. Using RSI, we introduce restricted strongly convex (RSC) functions.
\begin{definition}[Restricted strong convexity -- RSC($\nu$)]\label{rsc} A function $f(x):\RR^n\rightarrow \RR$ is restricted strongly convex with constant $\nu>0$ if it is convex,  has a finite minimizer, and satisfies RSI($\nu$).
\end{definition}
RSC is  weaker than strong convexity as \eqref{eq:rsi} is a relaxation to inequality \eqref{SC}. Some of our convergence results will be given for the following new classes of functions.

\begin{definition}[New function classes] Let $R, \nu>0$. Define function classes
\begin{align*}
\cL_{R}(\RR^n)&:=\{f:\RR^n\to\RR \mid f~\mbox{is convex and RLG}(R)\},\\
\cR_{R,\nu}(\RR^n)&:=\{f\in \cL_{R}(\RR^n)\mid f~\mbox{is RSC}(\nu)\},\\
\hat{\cR}_{L,\nu}(\RR^n)&:=\{f\in \cF_{L}(\RR^n)\mid f~\mbox{is RSC}(\nu)\}.
\end{align*}
\end{definition}

By definition, if $\mu\geq \nu$ and $L= R$, then we have
$$\mathcal{S}_{L,\mu}(\RR^n)\begin{array}{lcr}\subset &\mathcal{R}_{R,\nu}(\RR^n)&\subset\\
\subset &\hat{\cR}_{L,\nu}(\RR^n)\subset\mathcal{F}_{L}(\RR^n)&\subset \end{array}   \mathcal{L}_{R}(\RR^n).$$

Definition \ref{rsc} is different from another recent definition of restricted strong convexity from \cite{NRWY}.
\begin{definition}[Restricted strong convexity of \cite{NRWY}]\label{rsc2}
A function $f(x):\RR^n\rightarrow \RR$ satisfies the restricted strong convexity at $x_0$ with constants $\kappa_1, \kappa_2>0$ and tolerance function $r(x)$ if it is differentiable and
\begin{equation}
f(x_0+\delta)-f(x_0)-\langle f'(x_0),\delta\rangle \geq \kappa_1\|\delta\|^2-\kappa_2 (r(x_0))^2,
\end{equation}
for all $\delta \in \cC$, where $\cC$ is a certain point set.
\end{definition}
Definition \ref{rsc2} is a local and weakened version of strong convexity. With  $r(x)=0$ and $\mathbb{C}=\RR^n$, it reduces to the standard strong convexity. 

Many of the recent algorithms for sparse optimization are observed to converge quickly, at least on  problems that are not severely ``ill-conditioned''; however, their underlying objective functions are not strongly convex -- a property commonly used to ensure global linear convergence. When $A$ has more columns than rows, a function in the form of $g(Ax-b)$, even with a strongly convex function  $g$, is ``flat'' along many directions. Gradients along these directions are small, so minimization can progress very slowly. However, in problems with certain types of  $A$ and an additional regularization function $r(x)$ such as the $\ell_1$-norm, moving along these directions will significantly change   $r(x)$. We believe this has motivated the definition of restricted strong convexity in \cite{ANW}, which extends the ordinary definition by including the relaxation term involving  $r(x)$. That paper argues that, with high probability for problems with $A$ that is random or satisfies certain restricted eigenvalue   properties, Definition \ref{rsc2} is satisfied by  $f(x)= g(Ax-b)+r(x)$, and as a result, the prox-linear or gradient-projection iteration has a (nearly-)linear convergence behavior, specifically,
$$\|x^{(k+1)}-x^*\|^2\le c^k\|x^{(0)}-x^*\|^2 + o(\|x^*-x^o\|^2),$$
where $c<1$, $x^*$ and $x^o$ are the  minimizer and underlying true signal, respectively, and $x^{(k)}$ stands for the $k$th iterate.
Our paper focuses on the minimization of  convex differentiable functions in the general setting and establishes unmodified sublinear and linear convergence  without a probabilistic argument.

\subsection{Properties}
This subsection gives the core lemmas for establishing the main convergence results.
\begin{lemma}\label{lem01}
Let $\mathcal{X}^*$ be the nonempty solution set of \eqref{opt1}. If $f\in \cL_{R}(\RR^n)$ with $R>0$,
then we have

1) For any $(x,y) \in\Omega$ given in (\ref{Omeg}), it holds
\begin{equation}\label{Key01}
f(y)-f(x)-\langle \nabla f(x), y-x\rangle\leq \frac{R}{2}\|x-y\|^2;
\end{equation}

2) For any $y\in\mathcal{X}^*$, it holds
\begin{equation}\label{Key02}
\frac{1}{2R}\|\nabla f(x)\|^2\leq \langle \nabla f(x), x-y\rangle.
\end{equation}
\end{lemma}
\begin{proof}
For any $(x,y) \in\Omega$, \eqref{Key01} follows from
\begin{align*}
f(y) &=f(x)+\int_0^1 \langle\nabla f(x+\tau(y-x)), y-x\rangle d\tau\\
&=f(x)+\langle \nabla f(x), y-x\rangle + \int_0^1 \langle\nabla f(x+\tau(y-x))- \nabla f(x), y-x\rangle d\tau\\
 &\leq f(x)+\langle \nabla f(x), y-x\rangle + \int_0^1 \|\nabla f(x+\tau(y-x))- \nabla f(x)\| \| y-x\| d\tau\\
  &\leq f(x)+\langle \nabla f(x), y-x\rangle + \frac{R}{2}\|x-y\|^2,
\end{align*}
where the first inequality follows from the Cauchy-Schwartz inequality and the second one follows from the definition of  RLG. For part 2), for any $y\in\mathcal{X}^*$ we have
\begin{align*}
f^*=f(y) &\le f(x-R^{-1} \nabla f(x))\\
&\le f(x) + \langle \nabla f(x),(x-R^{-1}\nabla f(x))-x\rangle +\frac{R}{2}\| (x-R^{-1}\nabla f(x))-x\|^2\\
& = f(x) - (2R)^{-1}\|\nabla f(x)\|^2,
\end{align*}
where the second inequality follows from part 1).
Therefore, we have
$$\frac{1}{2R}\|\nabla f(x)\|^2\le f(x)-f(y)\le \langle \nabla f(x),x-y\rangle,$$
where the second inequality utilizes the convexity of $f$.
\end{proof}

Note that for general $y$, the inequality (\ref{Key02}) does not hold. For example, setting $y= x-\eta \nabla f(x)$ with $0<\eta<\frac{1}{2R}$ and assuming $\nabla f(x)\neq 0$ give $\langle \nabla f(x), x-y\rangle=\eta\cdot \|\nabla f(x)\|^2<\frac{1}{2R}\|\nabla f(x)\|^2$.

\begin{lemma}
Let $\mathcal{X}^*$ be the nonempty solution set of \eqref{opt1}. If $f\in \cR_{R,\nu}(\RR^n)$ with $R>0$ and $\nu>0$, then for every $\theta\in [0,1]$ the following holds:
\begin{equation}\label{Key}
\langle\nabla f(x)-\nabla f({x_\prj}), x-{x_\prj}\rangle \geq \frac{\theta}{2R}\|\nabla f(x)-\nabla f({x_\prj})\|^2+ (1-\theta)\nu \|x-{x_\prj}\|^2,
\end{equation}
where $x_\prj$ is the projection of $x$ onto the solution set $\mathcal{X}^*$.
\end{lemma}
\begin{proof}
Obviously, $x_\prj \in \cX^*$ and $\nabla f(x_\prj)=0$. Thus, from part 2) of Lemma \ref{lem01}, we have
\begin{equation}\label{Key1}
\langle\nabla f(x)-\nabla f({x_\prj}), x-{x_\prj}\rangle \geq \frac{1}{2R}\|\nabla f(x)-\nabla f({x_\prj})\|^2.
\end{equation}
On the other hand, from the definition of RSC($\nu$), we obtain
\begin{equation}\label{Key2}
\langle\nabla f(x)-\nabla f({x_\prj}), x-{x_\prj}\rangle \geq \nu \|x-{x_\prj}\|^2.
\end{equation}
Inequality \eqref{Key} follows from (\ref{Key1}) and (\ref{Key2}).
\end{proof}

Parameter $\theta$ in (\ref{Key}) will be optimized to obtain a convergence bound.

\begin{lemma}\label{lem2}
Let $f(x)$ satisfy  RSI$(\nu)$, $\nu>0$, and $\mathcal{X}^*$ be the nonempty solution set. For $\forall x\in \RR^{m}$ we have
\begin{equation}\label{RSC1}
f(x)-f({x_\prj})\geq \frac{\nu}{2}\|x-{x_\prj}\|^2,
\end{equation}
where $x_\prj$ is the projection of $x$ onto the solution set $\mathcal{X}^*$.
\end{lemma}
\begin{proof}
Since for any $\tau\in[0,1]$ point $y_\tau={x_\prj}+\tau(x-{x_\prj})\in\lfloor x,x_\prj \rfloor$ projects to $\mathcal{X}^*$ at  ${x_\prj}$,
we have
\begin{subequations}
\begin{align}
f(x) &=f({x_\prj})+\int_0^1 \langle\nabla f({x_\prj}+\tau(x-{x_\prj})), x-{x_\prj}\rangle d\tau\\
\label{eq1}
&=f({x_\prj})+\int_0^1 \frac{1}{\tau}\langle\nabla f({x_\prj}+\tau(x-{x_\prj}))-\nabla f({x_\prj}), \tau(x-{x_\prj})\rangle d\tau\\
\label{ineq02}
 &\geq f({x_\prj})+\int_0^1 \frac{1}{\tau}\nu \tau^2\|x-{x_\prj}\|^2d\tau\\
  &=f({x_\prj})+ \frac{\nu}{2}\|x-{x_\prj}\|^2
\end{align}
\end{subequations}
where (\ref{eq1}) follows from $\nabla f({x_\prj})=0$ and (\ref{ineq02})  from RSI$(\nu)$. \end{proof}

It is worth noting that since $x_\prj$ is restricted,  inequality (\ref{RSC1}) does not mean that $f$ grows \emph{everywhere} quicker than the quadratic function $q(x) = \frac{\nu}{2}\|x-{x_\prj}\|^2$.

\subsection{Examples of RSI and RSC functions}
%
%
\begin{figure}[ht]
\centering
\subfigure[]{
    \includegraphics[width=0.45\textwidth]{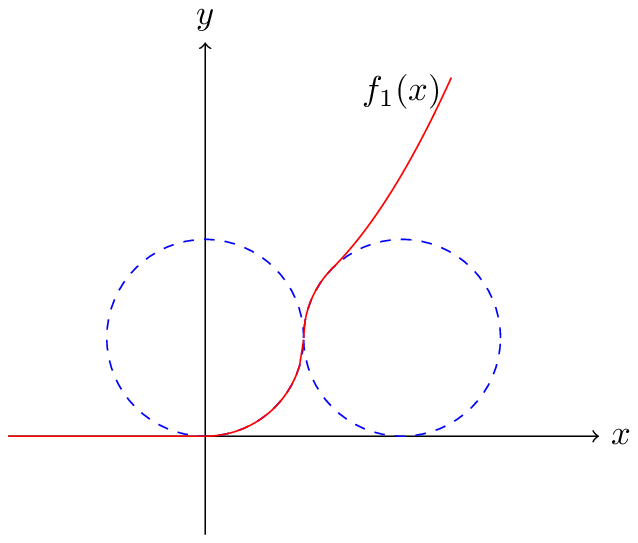}
    \label{fig:g1}
}
\subfigure[]{
   \includegraphics[width=0.45\textwidth]{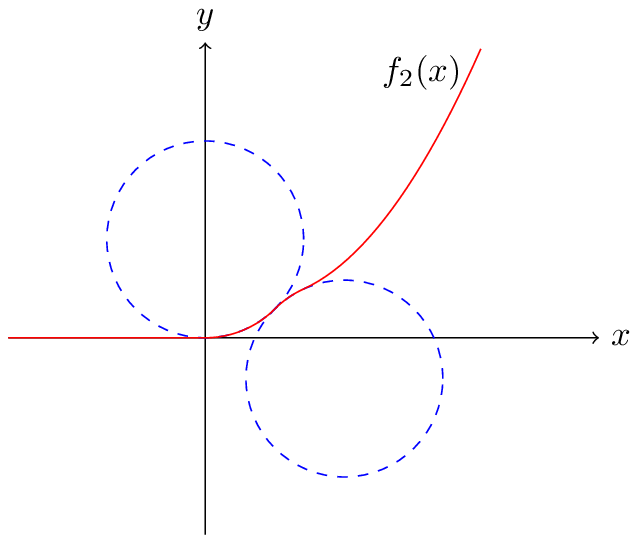}
    \label{fig:g2}
}
\caption[Optional caption for list of figures]{Non-convex  functions satisfying  RSI}
\label{fig01}
\end{figure}
Examples 1 and 2 below are non-convex and probably of no practical use. However, they illustrate that RSI inequality \eqref{eq:rsi} imposes a ``minimum average curvature'' of $f$ between $x$ and $x_\prj$, and unlike \eqref{SC}, it alone does \emph{not} guarantee convexity. Hence, the RSC definition must explicitly include convexity.
\begin{example}[Figure \ref{fig:g1}, RSI and non-convex]
\begin{equation}
f_1(x)=\left\{\begin{array}{ll}
0,& x\leq 0,\\
1-\sqrt{1-x^2},     & 0\leq x\leq 1,\\
1+\sqrt{1-(x-2)^2}, & 1\leq x\leq 2-\frac{\sqrt{2}}{2},\\
\frac{1}{2}(x-1+\frac{\sqrt{2}}{2})^2+\frac{1+\sqrt{2}}{2},& x\geq 2-\frac{\sqrt{2}}{2}.
\end{array}\right.
\end{equation}
 $f_1$ is non-convex, and its minimizer set is $(-\infty, 0]$. Since $f'_1(x)\to +\infty$ as $x\to 1$, $f'_1$ is not Lipschitz continuous. $f_1$ satisfies RSI($\nu$) with $\nu = \frac{2}{4-\sqrt{2}}=\min_{x\ge 0} f_1'(x)/x$.
\end{example}

\begin{example}[Figure \ref{fig:g2}, RSI and non-convex]
\begin{equation}
f_2(x)=\left\{\begin{array}{ll}
0,& x\leq 0,\\
1-\sqrt{1-x^2}, & 0\leq x\leq \frac{\sqrt{2}}{2},\\
\sqrt{1-(x-\sqrt{2})^2}-\sqrt{2}+1, & \frac{\sqrt{2}}{2}\leq x\leq 1,\\
\frac{1}{2}(x-1+\sqrt{\frac{\sqrt{2}-1}{2}})^2+ \sqrt{2\sqrt{2}-2}+\frac{5-5\sqrt{2}}{4}, & x\geq 1.
\end{array}\right.
\end{equation}
$f_2$ is non-convex, and its minimizer set is $(-\infty, 0]$. Unlike $f_1$,
$\max_{x\geq 0}\frac{\nabla f_2(x)}{x}$
is finite and thus $f_2$ has  a Lipschitz continuous gradient. $f_2$ satisfies RSI($\nu$) with $\nu = \sqrt{\frac{\sqrt{2}-1}{2}}=\min_{x\ge 0} f_2'(x)/x$.
\end{example}
Examples 3 and 4 below explain that RSC and strict convexity do not contain each other, and strong convexity is strictly included in their intersection. Recall that a function $f$ is strictly convex if $f(\alpha x+(1-\alpha)y)< \alpha f(x) + (1-\alpha) f(y)$  for any $x\not=y$ and $\alpha\in (0,1)$.
\begin{figure}[ht]
\begin{center}\includegraphics[width=0.4\textwidth]{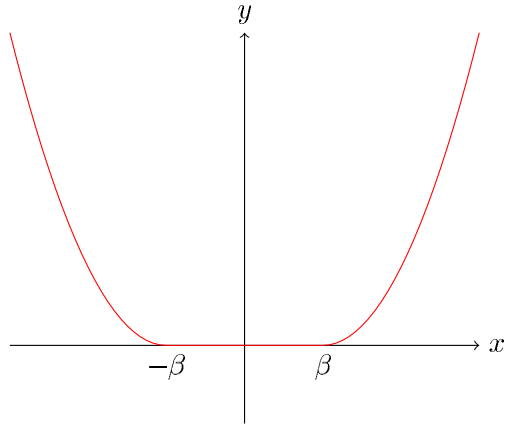}\end{center}
\caption{RSC but not strictly convex}\label{fig:shrink}
\end{figure}


\begin{example}[Figure \ref{fig:shrink}, RSC but not strictly convex] Let $x\in\RR, \beta >0$ and define
\begin{align}
\label{shrk}
\shrink_\beta(x) & = \sign(x)\max\{|x|-\beta,0\},\\
f_3(x)& =\frac{1}{2}\|\shrink_\beta(x)\|^2_.\nonumber
\end{align}
$f$ is not strictly convex since $f_3(x)= 0$ for $x\in\cX^*=[-\beta, \beta]$, which is its minimizer set. On the other hand, $f_3(x) = (1/2)\|x-\beta\|^2$ for $x\ge \beta$ and $f_3(x)=(1/2)\|x+\beta\|^2$ for $x\le \beta$, so $f_3$ is RSC($\nu$) with $\nu = 1$.
\end{example}

\begin{example}[Strictly convex, but not RSC]  Functions $f(x)=x^4$ and $f(x)=e^x$ are  strictly convex but \emph{not} RSC. In particular, $f(x)=e^x$ does not have a minimizer though it is lower bounded by 0.
\end{example}
\begin{figure}[ht]
\begin{center}\includegraphics[width=0.45\textwidth]{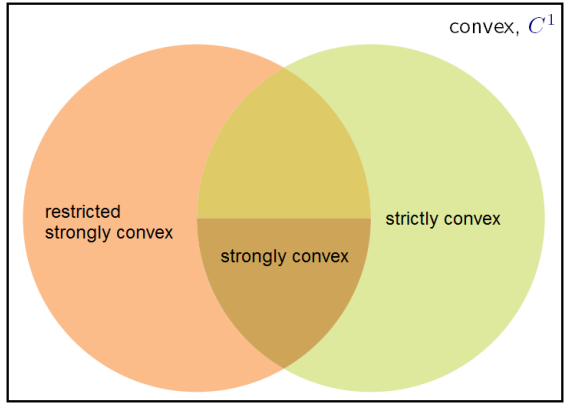}\end{center}
\caption{Classes of convex differentiable functions}\label{fig:comp}
\end{figure}
Motivated by the above examples, we can divide convex differentiable functions into subclasses of RSC, strictly convex, and strongly convex functions depicted in Figure \ref{fig:comp}. Strictly and strongly convex functions do not need to be differentiable. Although our definition of RSC can be generalized for non-differentiable functions through their subdifferentials, we keep it simple as is.
\begin{example}[Dual objective of augmented $\ell_1$ model]\label{eg:augl1} Let $A\in\RR^{m\times n}$. The Lagrange dual problem to
\beq\label{augl1}\min \left\{\|x\|_1 + \frac{1}{2\alpha}\|x\|^2:Ax = b\right\}
\eeq
is
\beq\label{augl1d}\max_y f(y) = b^T y - \frac{\alpha}{2}\|\shrink_1(A^Ty)\|^2, \eeq
where $\shrink_1(z)$ is given in \eqref{shrk}.
Provided that $Ax=b$ is consistent,  \emph{\cite{LY}} shows that $-f$ is  RSC($\nu$) with $\nu >0$. (See Lemma 7 of  \emph{\cite{LY}} for an explicit lower bound of $\nu$).
\end{example}
Admittedly, establishing RSC and deriving a bound for $\nu$ are not straightforward as they typically involve projection to the minimizer set $\cX^*$, which may not be easy to analytically derive. On the other hand, we have to live with RSC as we will show later that it is both sufficient and necessary. Next we present some results of deriving RSC for certain composite functions.
\begin{theorem}[Linear composition 1]\label{thm:gax} Let $g\in\hat{\cR}_{L,\nu}(\RR^m)$. If $g$ has a unique minimizer $y^*$ and matrix $A\in\RR^{m\times n}$ $(m\leq n)$ has full row-rank (i.e., $A$ is surjective), then function $f(x)=g(Ax)$ is RSC. Specifically,
\beq\label{gax}
f(x)\in\hat{\cR}_{\bar L, \bar{v}}(\RR^m),
\eeq
where
$\bar{L}=L\|A\|^2$ and $\bar{\nu}=\nu\lambda_{\min}(AA^T)$.
\end{theorem}
Applying this theorem, any strongly convex function $g$ with Lipschitz continuous gradient satisfies the condition of Theorem \ref{thm:gax} and thus $f(x) = g(Ax)$ is RSC if $A$ has full row-rank though $f$ is generally not strongly convex. ($f$ will be strongly convex if $A$ has full column-rank, following a standard argument). 
$f(x)=g(Ax)$ arises in various applications including examples in convex quadratic minimization,  statistical regression, routing problems in data networks, and many others. 

\begin{proof}[Proof of Theorem \ref{thm:gax}] For any $x,y\in\RR^n$, we have
$$\|\nabla f(x) - \nabla f(y)\| = \| A^T \nabla g(Ax) -  A^T \nabla g(Ay)\| \le L\|A\| \|A(x - y)\| \le (L \|A\|^2) \|x - y\|,$$
which means $f\in \cF_{\bar{L}}$.
By definition, the minimizer set of $f$ is
$$\cX^*=\{x\in\RR^n:Ax = y^*\},$$
which is nonempty since $A$ has full row-rank. The projection of any $x\in\RR^n$ to $\cX^*$ is $$x_\prj =x+A^T(AA^T)^{-1} (y^*-Ax) .$$
Since $\nabla f(x) = A^T\nabla g(Ax)$, we
$$ \langle \nabla f(x),x-x_\prj\rangle = \langle  \nabla g(Ax) -  \nabla g(Ax_\prj), Ax - A x_\prj\rangle  \ge \nu \|A(x - x_\prj)\|^2 \ge (\nu \lambda_{\min}(AA^T)) \|x - x_\prj\|^2.$$
where the first inequality follows from $g\in\cR_{L,\nu}$ and the second one from $x-x_\prj\in\mathrm{Range}(A^T)$.
\end{proof}

Next, we show that if function $g$ is strictly convex, then we no longer need $A$ to have full row-rank. We first present two lemmas:
\begin{lemma}[\cite{t1}]\label{lemadd1}
Let $f(x)=g(Ax)$ and assume that $g$ is strictly convex and the minimizer set of $f$, denoted by $\cX^*$, is nonempty. Then, there exists a vector $t^*\in\RR^m$ such that  $\cX^*=\{x\in\RR^n: Ax = t^*\}$.
\end{lemma}

\begin{lemma}[\cite{LY}]\label{lemadd2}
Let $\lambda^{++}_{\min}$ denote the minimum \emph{strictly positive} eigenvalue of a nonzero symmetric matrix
$S$, assuming its existence. Namely, given $\{\lambda_i(S)\}$, the set of eigenvalues of $S$,
$$\lambda^{++}_{\min}(S)= \min\{\lambda_i(S) : \lambda_i(S) > 0\}.$$
Then, for every nonzero matrix $A$, we have
$$\lambda^{++}_{\min}(AA^T)= \min_{\|A\alpha\|_2=1} (A\alpha)^T(AA^T)(A\alpha).$$
\end{lemma}

Furthermore, we need the sets $\mathcal{B}_1(t^*, \gamma)=\{y\in\RR^m: \|y-t^*\|_2\leq \gamma\}$
and  $\mathcal{B}_2(t^*, \gamma, A)=\{x\in\RR^n: Ax\in\mathcal{B}_1(t^*, \gamma) \}$. Now, let us state the result
\begin{theorem}[Linear composition 2]\label{lncp2} Assume that $g$ is strongly convex with  modulus $\mu$ on $\mathcal{B}_1(t^*, \gamma)$ for some $\gamma>0$ and $f(x)=g(Ax)$ has a minimizer. Then, $f$ satisfies the RSI with $\mu \lambda^{++}_{\min}(A^TA)$ for all $x\in \mathcal{B}_2(t^*, \gamma, A)$. 

In addition, if $\nabla g$ is Lipschitz continuous with constant $L$, then \beq\label{gbx}
f(x)\in\hat{\cR}_{\bar L, \bar{v}}(\RR^m),
\eeq
where
$\bar{L}=L\|A\|^2$ and $\bar{\nu}=\mu \lambda^{++}_{\min}(A^TA)$.
\end{theorem}
\begin{proof}  For $x\in \mathcal{B}_2(t^*, \gamma, A)$, let $x_\prj$ be its projection onto $\cX^*$. Since $\cX^*$ is nonempty, there exists $x^*\in\cX^*$ such that $Ax^*=t^*$ by Lemma \ref{lemadd1}. By the definition of projection, we have
$$x_\prj=\arg\min\frac{1}{2}\|x-z\|_2^2, ~\textrm{subject to}~~Az=Ax^*$$
Hence, there exists a Lagrange multiplier $\lambda$ such that $x-x_\prj=A^T\lambda\in\mathrm{Range}(A^T)$.

Since $\nabla f(x) = A^T\nabla g(Ax)$, we have
$$ \langle \nabla f(x)-\nabla f(x_\prj),x-x_\prj\rangle = \langle  \nabla g(Ax) -  \nabla g(Ax_\prj), Ax - A x_\prj\rangle  \ge \mu \|A(x - x_\prj)\|^2 \ge (\mu \lambda^{++}_{\min}(A^TA)) \|x - x_\prj\|^2.$$
where the first inequality follows from that $g$ is strongly convex with  modulus $\mu$ on $\mathcal{B}_1(t^*, \gamma)$ and $x\in \mathcal{B}_2(t^*, \gamma, A)$, and the second one follows from Lemma \ref{lemadd2} and the fact that $x-x_\prj\in\mathrm{Range}(A^T)$. Then \eqref{gbx} follows trivially.
\end{proof}

Applying the Cauch-Schwartz inequality to $ \langle \nabla f(x)-\nabla f(x_\prj),x-x_\prj\rangle $, it is easy to see that Theorem \ref{lncp2} immediately implies that $\|x - x_\prj\|\leq (\mu \lambda^{++}_{\min}(A^TA)) ^{-1}\|\nabla f(x)\|$, which is referred to as the error bound condition and is a key to the analysis in \cite{s}.

\subsection{Convex conjugacy}
The conjugate of convex function $f$ is
\beq\label{congj}
f^*(y) := \sup_x\{\langle y, x\rangle -f(x)\}.
\eeq
A duality relation can be obtained between RLG and RSC, in analogy to the well-known result that a convex function $f$ is differentiable and $\nabla f$ is Lipschitz-continuous with constant $L$ if and only if $f^*$ is strongly convex with constant $1/L$. In this subsection, we consider non-differentiable functions to present our result (while we restrict ourselves to differentiable functions in other sections).
\begin{definition}\label{rlsub}
Let $f$ be a convex function. We say that $f$ has restricted Lipschitz \emph{subgradients} if there exists $L>0$ such that for any $x\not=0$,$$L\langle p - q,x\rangle \ge \| p-q\|^2,\quad \forall p\in \partial f(x), ~q=\Proj_{\partial f(0)}(p).$$
\end{definition}
Definition \ref{rlsub} applies to non-differentiable functions while the usual Lipschitz continuity of gradient of course requires differentiability. In Example \ref{eg:augl1}, the primal objective \eqref{augl1} is non-differentiable but satisfies Definition \ref{rlsub} with $L = \alpha^{-1}$.

\begin{theorem}
Let $f$ be a strictly convex function and $0\in \dom f$. $f$ has  restricted Lipschitz subgradients with constant $L>0$ if and only if $f^*$ is RSC with constant $L^{-1}>0$.
\end{theorem}
\begin{proof}Due to the strict convexity of $f$, the sup-problem in \eqref{congj} has a unique solution, denoted by $x(y)$, which satisfies
$$0\in y -\partial f(x(y)). $$
Also, $f^*$ is differentiable since $f$ is strictly convex, and  $\nabla f^*(y) = x(y)$.

Consider problem $\min f^*(y)$, which has solution set  $\mathcal{Y}^*=\{y:\nabla f^*(y)=0\}=\{y: x(y)=0\}=\partial f(0)$.

 ``$\Longrightarrow$''
Pick $y\not\in \mathcal{Y}^*$ and let $y_{\prj}=\Proj_{\mathcal{Y}^*}(y)=\Proj_{\partial f(0)}(y)\in \mathcal{Y}^*$. From $y\in \partial f(x(y))$,
$$
\langle \nabla f^*(y)-\nabla f^*(y_{\prj}), y-y_{\prj}\rangle  = \langle x(y),y-y_{\prj}\rangle\ge L^{-1} \|y-y_{prj}\|^2,$$
where the last inequality follows from Definition \ref{rlsub}.

``$\Longleftarrow$'' Pick any $x\not=0$ and $p\in \partial f(x)$. Let  $y=p$ and $y_\prj=q=\Proj_{\partial f(0)}(p)$. Then, $\nabla f^*(y) =x$ and $\nabla f^*(y_\prj) = 0$. Then,
$$L\langle p - q,x\rangle =L\langle y - y_\prj, \nabla f^*(y)- \nabla f^*(y_{\prj})\rangle\ge \|y-y_{prj}\|^2 = \|p-q\|^2, $$
where the inequality follows from the definition of RSC.
\end{proof}

\section{Main results}\label{sc:main}
\begin{table}[ht]
\begin{center}\begin{tabular}{c|c|c|c}\hline
function & 1st-order oracle & ordinary gradient & accelerated gradient\\
class & lower bound & method & method\\\hline
~&~ &~ & ~\\
$\cL_R(\RR^n)$ & $O\left(\sqrt\frac{R}{{\epsilon}}\right)$ & Theorem \ref{thm:rlgsub}: $O\left(\frac{R}{\epsilon}\right)$ &   Theorem \ref{thm03}: $O\left(\sqrt\frac{R}{{\epsilon}}\right)$ \\[15pt]
$\cR_{R,\nu}(\RR^n)$ & $O\left(\sqrt{\frac{R}{\nu}}\log\frac{1}{\epsilon}\right)$ & Theorem \ref{thm2R}:  $O\left(\frac{R}{\nu}\log\frac{1}{\epsilon}\right)$ &  Theorem \ref{thm3}:  $O\left(\sqrt{\frac{R}{\nu}}\log\frac{1}{\epsilon}\right)$\\[10pt]\hline
\end{tabular}\end{center}
\caption{Complexities of the new classes of functions}\label{complexity2}
\end{table}
This section derives the complexity bounds for the ordinary and accelerated gradient methods under RLG and/or RSC conditions; the derived complexities are summarized in Table 2. The bounds are presented for the following error quantities:
\begin{enumerate}
\item Objective error: $\Delta_k := f(x^{(k)})-f^*$, where $f^*=\min_{x\in\RR^n}f(x)$;
\item Solution error: $r_k :=\|x^{(k)}-x^{(k)}_\prj\|=\min\{\|x^{(k)}-x^*\|:x^*\in \cX^*\}$.
\end{enumerate}
\subsection{Ordinary gradient descent}
\begin{algorithm}[htb]
\caption{Ordinary gradient descent method}  \label{alg0}
\begin{tabbing}
\textbf{Input:} Initialize $x^{(0)}\in \RR^n$ and select stepsize $h>0$.\\

1: \textbf{for} $k=0, 1, \cdots,$ \textbf{do}\\

2: \quad $x^{(k+1)}=x^{(k)}-h \,\nabla f(x^{(k)})$;\\

3: \textbf{end for}
\end{tabbing}
\vspace{-10pt}
\end{algorithm}

\begin{theorem}[Sublinear convergence for $\cL_R(\RR^n)$]\label{thm:rlgsub} Assume that in problem \eqref{opt1}, $f\in \cL_R(\RR^n)$ with $R>0$. Then Algorithm \ref{alg0} with stepsize $h\in(0, 1/R]$ converges sublinearly with
$$\Delta_k=O(\frac{R\, r_0^2}{k}),$$
where $r_0=\|x^{(0)}-x^{(0)}_\prj\|$.
It reaches  $\epsilon$-accuracy (i.e., $\Delta_k<\epsilon$) in $O(\frac{R}{\epsilon})$ iterations.
\end{theorem}
\begin{proof}
Firstly, we prove that $r_k$ is non-increasing and thus uniformly bounded by $r_0$. From part 2) of Lemma \ref{lem01} and $h=\alpha/R$, where $\alpha\in(0,1]$, we have
$$ h^2\|  \nabla f(x^{(k)}) \|^2 =2\alpha h\cdot \frac{1}{2R}\|\nabla f(x^{(k)})\|^2\le 2\alpha h \langle \nabla f(x^{(k)}),x^{(k)}-x^{(k)}_\prj\rangle\le 2 h \langle \nabla f(x^{(k)}),x^{(k)}-x^{(k)}_\prj\rangle, $$
so in turn we get from $x^{(k+1)} = x^{(k)} -h  \nabla f(x^{(k)})$ that
\begin{subequations}\label{repeat0}
\begin{align}
r_{k+1}^2 =\|x^{(k+1)}-x^{(k+1)}_\prj \|^2 &\leq  \|x^{(k+1)}-x^{(k)}_\prj \|^2\\
&= \|x^{(k)}-x^{(k)}_\prj -h \nabla f(x^{(k)})\|^2\\
&= \|x^{(k)}-x^{(k)}_\prj\|^2  -2h \langle\nabla f(x^{(k)}), x^{(k)}-x^{(k)}_\prj\rangle+ h^2\|\nabla f(x^{(k)})\|^2\le r_k^2
\end{align}
\end{subequations}
and $r_k\le r_0$, $\forall k$.

Next, by the convexity of $f$, $\langle\nabla f(x^{(k)}),x^{(k)}-x^*\rangle\ge f(x^{k})- f^*\ge 0 $. Since $r_k\le r_0$, we have the bound$$\|\nabla f(x^{(k)})\| \ge \frac{r_k}{r_0}\|\nabla f(x^{(k)})\| \ge\frac{ |\langle\nabla f(x^{(k)}),x^{(k)}-x^*\rangle|}{r_0} \ge\frac{\Delta_k}{r_0}. $$
By part 1) of Lemma \ref{lem01}, we have
\begin{align*}
    \Delta_{k+1}& \le \Delta_k + \langle \nabla f(x^{(k)}),x^{(k+1)}-x^{(k)}\rangle +\frac{R}{2}\|x^{(k+1)}-x^{(k)}\|^2\\
    & = \Delta_k - h(1-\frac{h R}{2})   \|\nabla f(x^{(k)})\|^2\\
    & \le \Delta_k - \frac{h}{r_0^2}(1-\frac{h R}{2})  \Delta_k^2.
\end{align*}
For $h=\alpha/R$, where $0<\alpha\le 1$,  $\frac{h}{r_0^2}(1-\frac{h R}{2})=\frac{\alpha(2-\alpha)}{2 (Rr_0^2)}=O(\frac{1}{Rr_0^2}).$ Dividing the both sides of $\Delta_{k+1}\le \Delta_k - O(\frac{1}{Rr_0^2})\Delta_k^2$ by $\Delta_k \Delta_{k+1}$,  we get $ (1/\Delta_{k+1})\ge (1/\Delta_k)+O(\frac{1}{Rr_0^2})$. Therefore, $\Delta_k=O(R\, r_0^2/k)$, following from which $\Delta_k<\epsilon$ is guaranteed in $O(Rr_0^2/\epsilon)=O(R/\epsilon)$ iterations.
\end{proof}

(Restricted) Lipschitz continuity of $\nabla f$ alone cannot provide a decay rate for $r_k$. In fact, $r_k$ can decay arbitrarily slowly as function $f$ becomes arbitrarily close to being flat near its minimizer. With the addtional RSC assumptions, the theorems  below  give geometrically-decaying bounds for both $r_k$ and $\Delta_k$.
\begin{theorem}[linear convergence for $\cR_{R,\nu}$]\label{thm2R}
Assume that in problem \eqref{opt1}, $f\in \cR_{R,\nu}(\RR^n)$ with some $R,\nu>0$. Then Algorithm \ref{alg0} with stepsize $h=\frac{1}{2R}$  converges linearly with
\begin{align*}
    r_{k+1} & \le (1-\frac{\nu}{2R})^{1/2} \cdot r_k,\\
    \Delta_k& \le \frac{R}{2} r_0^2 (1-\frac{\nu}{2R})^k.
\end{align*}
    It reaches  $\epsilon$-accuracy in $O\left(\frac{R}{\nu}\log\frac{1}{\epsilon}\right)$ iterations.

Conversely, assuming that $f$ has the unique solution $x^*$ and Algorithm starts from arbitrary $x^{(0)}$ has a finite stepsize $h$, linear convergence in the form of $\|x^{(k+1)}-x^* \|^2 \leq  (1-\delta) \|x^{(k)}-x^*\|^2$ for some $0<\delta<1$ requires  $f$ to be RSC($\nu$) for some $\nu>0$.
\end{theorem}

\begin{proof}
Recall that $x^{(k)}_\prj$ is the projection of $x^{(k)}$ onto the solution set $\mathcal{X}^*$ and $r_k=\|x^{(k)}-x^{(k)}_\prj \|$. Thus, $\nabla f(x^{(k)}_\prj)=0$.
For every $\theta\in [0,1]$ we have
\begin{subequations}\label{repeat}
\begin{align}\label{ineq0}
\|x^{(k+1)}-x^{(k+1)}_\prj \|^2
&\leq \|x^{(k)}-x^{(k)}_\prj\|^2  -2h \langle\nabla f(x^{(k)}), x^{(k)}-x^{(k)}_\prj\rangle+ h^2\|\nabla f(x^{(k)})-\nabla f(x^{(k)}_\prj)\|^2\\
\label{ineq1}
&\leq \|x^{(k)}-x^{(k)}_\prj\|^2  -2h(\frac{\theta}{2R}\|\nabla f(x^{(k)})-\nabla f(x^{(k)}_\prj)\|^2+ (1-\theta)\nu \|x^{(k)}-x^{(k)}_\prj\|^2)\\ \nonumber
&~ ~+ h^2\|\nabla f(x^{(k)})-\nabla f(x^{(k)}_\prj)\|^2\\
\label{ineq2}
&= (1-2(1-\theta)\nu h)\|x^{(k)}-x^{(k)}_\prj\|^2+(h^2-\frac{\theta h}{R})\|\nabla f(x^{(k)})-\nabla f(x^{(k)}_\prj)\|^2,
\end{align}
\end{subequations}
 where inequality (\ref{ineq0}) follows from (\ref{repeat0}) and inequality (\ref{ineq1}) utilizes (\ref{Key}). We minimize (\ref{ineq2})  over  $\theta$ and  $h$ and obtain $\theta=\frac{1}{2}$ and $h=\frac{1}{2R}$; the details can be found in Appendix. Then from (\ref{ineq2}) we get
\begin{equation}
\|x^{(k+1)}-x^{(k+1)}_\prj \|^2 \leq (1-\frac{\nu}{2R})\|x^{(k)}-x^{(k)}_\prj \|^2,
\end{equation}
i.e., $r_{k+1}  \le (1-\frac{\nu}{2R})^{1/2} \cdot r_k$.

By part 1) of Lemma \ref{lem01}, $\nabla f(x^{(k)}_\prj)=0$, and $r_{k+1}  \le (1-\nu/2R)^{1/2} \cdot r_k$,  we derive that
\begin{equation}
\Delta_k =f(x^{(k)})-f^*\leq \frac{R}{2} \|x^{(k)}-x^{(k)}_\prj\|^2 = \frac{R}{2} r_k^2 \leq \frac{R}{2}r_0^2(1-\frac{\nu}{2R})^{k},
\end{equation}
which shows $\Delta_k \le \frac{R}{2}r_0^2 (1-\frac{\nu}{2R})^k$, following from which $\Delta_k<\epsilon$ is guaranteed in $O\left(\frac{R}{\nu}\log\frac{1}{\epsilon}\right)$ iterations.

Now, we show the converse result. Since $f$ has the unique solution $x^*$, we have $x^{(k+1)}_\prj=x^{(k)}_\prj=x^*$. Noticing $x^{(k+1)}=x^{(k)}-h \nabla f(x^{(k)})$, we get
$$\|x^{(k+1)}-x^* \|^2 =  \|x^{(k)}-x^*\|^2  -2h \langle\nabla f(x^{(k)}), x^{(k)}-x^*\rangle+ h^2\|\nabla f(x^{(k)})-\nabla f(x^*)\|^2.$$ From $\|x^{(k+1)}-x^* \|^2 \leq  (1-\delta) \|x^{(k)}-x^*\|^2$ for some $0<\delta<1$, we have
$$h^2\|\nabla f(x^{(k)})-\nabla f(x^*)\|^2-2h \langle\nabla f(x^{(k)}), x^{(k)}-x^*\rangle\leq -\delta \|x^{(k)}-x^*\|^2,$$
and consequently $\langle\nabla f(x^{(k)}), x^{(k)}-x^*\rangle \geq \frac{\delta}{2h}\|x^{(k)}-x^*\|^2$ after dropping $h^2\|\nabla f(x^{(k)})-\nabla f(x^*)\|^2\ge 0$. As  $x^{(0)}$ is  arbitrary,  $f$ is RSC$(\nu)$ with $\nu=\frac{\delta}{2h}>0$.
\end{proof}

If RLG is strengthened to global Lipschitz continuity, we can take a \emph{possibly} larger stepsize $1/L$ instead of $1/(2R)$ and have \emph{possibly} better constants in the bound as follows.
\begin{theorem}[Linear convergence for $\hat{\cR}_{L,\nu}$]\label{thm:RSC+L} Assume that in problem \eqref{opt1},  $\nabla f$ is $L$-Lipschitz continuous and $f$ is RSC($\nu$) with $L,\nu>0$.  Then Algorithm \ref{alg0} with  stepsize $h=1/L$  converges linearly with
\begin{align*}
    r_{k+1} & \le (1-\nu/L)^{1/2} \cdot r_k,\\
    \Delta_k & \le \frac{L}{2}r_0^2 (1-\nu/L)^k. \end{align*}
    It reaches  $\epsilon$-accuracy in  $O\left(\frac{L}{\nu}\log\frac{1}{\epsilon}\right)$ iterations.
\end{theorem}
\begin{proof}
By replacing Lemma \ref{lem01} with the following two Lemmas and repeating the arguments in Theorem \ref{thm2R}, the desired linear convergence rates can be derived.
\end{proof}

\begin{lemma}[\cite{n2} Theorem 2.1.5]\label{lem1}
If $f(x)\in \mathcal{F}_{L}(\RR^n)$, it obeys
\begin{align}
f(x)\leq f(y)+ \langle \nabla f(y), x-y\rangle +\frac{L}{2}\|x-y\|^2,&\quad\forall x, y\in\RR^n;\\
\langle \nabla f(x)-\nabla f(y), x-y\rangle \geq \frac{1}{L}\|\nabla f(x)-\nabla f(y)\|^2,&\quad\forall x, y\in\RR^n \label{ine:c2}.
\end{align}
\end{lemma}

\begin{lemma}
Let $\mathcal{X}^*$ be the nonempty solution set of \eqref{opt1}. If $\nabla f$ is $L$-Lipschitz continuous and  $f$ is RSC($\nu$) with  $L,\nu>0$, then for every $\theta\in [0,1]$ the following holds:
\begin{equation}\label{Key3}
\langle\nabla f(x)-\nabla f({x_\prj}), x-{x_\prj}\rangle \geq \frac{\theta}{L}\|\nabla f(x)-\nabla f({x_\prj})\|^2+ (1-\theta)\nu \|x-{x_\prj}\|^2,
\end{equation}
where $x_\prj$ denotes the projection of $x$ onto the solution set $\mathcal{X}^*$.
\end{lemma}
\begin{proof}
Inequality (\ref{Key3}) follows from inequalities (\ref{eq:rsi}) and (\ref{ine:c2}).
\end{proof}

\subsection{Accelerated gradient descent}
\begin{algorithm}
\caption{Nesterov's accelerated gradient method} \label{alg1}
\begin{tabbing}
\textbf{Input:} Initialization $y^{(0)}\in \RR^n, \theta_0=1$, and $h>0$.\\

1: \textbf{for} $k=0, 1, \cdots,$ \textbf{do}\\

2: \quad $x^{(k+1)}=y^{(k)}-h\nabla f(y^{(k)})$; \quad (negative gradient step)\\

3: \quad $\beta_{k+1}=(1-\theta_k)(\sqrt{\theta_k^2+4}-\theta_k)/2$; \quad (extrapolation weight)\\

4: \quad $y^{(k+1)}=x^{(k+1)}+\beta_{k+1}(x^{(k+1)}-x^{(k)})$; \quad (extrapolation)\\

5: \quad $\theta_{k+1}=\theta_k(\sqrt{\theta_k^2+4}-\theta_k)/2$; \quad (dampening of acceleration parameter)\\

6: \textbf{end for}
\end{tabbing}\vspace{-10pt}
\end{algorithm}
Algorithm \ref{alg1} is equivalent to  Constant Step Scheme II on Page 80 of \cite{n2} (their $\alpha_k\equiv \theta_k$,  their $q= 0$) and FISTA on Page 193 of \cite{BT} without the nonsmooth regularization function $g$ (their $t_k\equiv 1/\theta_k$\footnote{Step 5 of Algorithm \ref{alg1} satisfies $\theta_{k+1}^2=(1-\theta_{k+1})\theta_k^2$; plugging $\theta_k=1/t_k$ and $\theta_{k+1}=1/t_{k+1}$, we obtain $t_{k+1}^{-2}=(1-t_{k+1}^{-1})t_k^{-2}$, which gives step (4.2) in \cite{BT}. Also, $\beta_{k+1}$ equals $\frac{t_k-1}{t_{k+1}}$ in (4.3).}).
    \begin{theorem}\label{thm03}
    Assume that in problem \eqref{opt1}, $f\in \cL_R(\RR^n)$ with $R>0$. Then Algorithm \ref{alg1} with $h=1/R$ converges sublinearly with
\begin{equation}\label{ord1}
\Delta_k\leq \frac{4R\cdot\| x^{(1)}- x^{(1)}_\prj\|^2}{(k+1)^2}.
\end{equation}
It reaches  $\epsilon$-accuracy in $O(\sqrt{\frac{R}{\epsilon}})$ iterations.
\end{theorem}
The proof below is self-contained and inspired by  \cite{t2}. Its $O(\sqrt{\frac{R}{\epsilon}})$ is better than $O(\frac{R}{\epsilon})$ of Theorem \ref{thm:rlgsub}.
\begin{proof}
Sequences $\{\theta_k\}$ and $\{\beta_k\}$ obey the following recursive relationships:
$$\frac{1}{\theta^2_k}=\frac{1-\theta_{k+1}}{\theta^2_{k+1}}\quad\mbox{and}\quad\beta_{k+1}=\theta_k(1-\theta_k)/(\theta_k^2+\theta_{k+1})=\theta_{k+1}(\frac{1}{\theta_k}-1).$$
Defining $x^{(0)}=0$ and $v^{(k+1)}=x^{(k)}+\frac{1}{\theta_k}(x^{(k+1)}-x^{(k)})$, we can rewrite  $y^{(k+1)}=\theta_{k+1}v^{(k+1)}+(1-\theta_{k+1})x^{(k+1)}$. From part 1) of Lemma \ref{lem01} and the convexity of $f$, for any $z\in \RR^n$ we have
\begin{subequations}
\begin{align*}
f(x^{(k+1)}) &\leq  f(y^{(k)})+\langle \nabla f(y^{(k)}), x^{(k+1)}-y^{(k)}\rangle+\frac{R}{2}\| x^{(k+1)}-y^{(k)}\|^2 \\
&\leq  (f(z)+\langle \nabla f(y^{(k)}), y^{(k)}-z \rangle) +\langle \nabla f(y^{(k)}), x^{(k+1)}-y^{(k)}\rangle +\frac{R}{2}\| x^{(k+1)}-y^{(k)}\|^2 \\
&\leq  f(z)+\langle \nabla f(y^{(k)}), x^{(k+1)}-z\rangle +\frac{R}{2}\| x^{(k+1)}-y^{(k)}\|^2 \\
&\leq  f(z)+R \langle  x^{(k+1)}-y^{(k)}, z- x^{(k+1)}\rangle +\frac{R}{2}\| x^{(k+1)}-y^{(k)}\|^2.
\end{align*}
\end{subequations}
Setting $z=\theta_k x^*+(1-\theta_k)x^{(k)}$, where $x^*\in \mathcal{X}^*$, and using the convexity of $f$, we get
\begin{equation} \label{eq:bound}
f(x^{(k+1)}) \leq \theta_k f^*+(1-\theta_k)f(x^{(k)})+R \langle  x^{(k+1)}-y^{(k)}, \theta_k x^*+(1-\theta_k)x^{(k)}- x^{(k+1)}\rangle +\frac{R}{2}\| x^{(k+1)}-y^{(k)}\|^2.
\end{equation}
Since $\theta_k x^*+(1-\theta_k)x^{(k)}- x^{(k+1)}=\theta_k(x^*-v^{(k+1)})$ and $x^{(k+1)}-y^{(k)}=\theta_k(v^{(k+1)}-v^{(k)})$, we have
\begin{align*}
R \langle  x^{(k+1)}-y^{(k)}, \theta_k x^*+(1-\theta_k)x^{(k)}- x^{(k+1)}\rangle =&R\theta_k^2 \langle v^{(k+1)}-v^{(k)}, x^*-v^{(k+1)}\rangle \\
=&R \theta_k^2 \langle v^{(k+1)}-x^*, v^{(k)}- x^*\rangle-R\theta_k^2 \| v^{(k+1)}-x^*\|^2
\end{align*}
and
\begin{equation}\label{eq:norm}
\frac{R}{2}\| x^{(k+1)}-y^{(k)}\|^2=\frac{R\theta_k^2 }{2}(\| v^{(k+1)}-x^*\|^2 +\| v^{(k)}- x^*\|^2- 2\langle v^{(k+1)}-x^*, v^{(k)}- x^*\rangle).
\end{equation}
Substituting these  equations into the last two terms of (\ref{eq:bound}), we get
\begin{equation} \label{eq:bound1}
f(x^{(k+1)}) \leq \theta_k f^*+(1-\theta_k)f(x^{(k)}) - \frac{R\theta_k^2 }{2}\| v^{(k+1)}-x^*\|^2 +\frac{R\theta_k^2 }{2} \| v^{(k)}- x^*\|^2.
\end{equation}
Reordering the terms and dividing by $\theta_k^2$ and then recursively deducing, we have
\begin{subequations}\label{eq:bound2}
\begin{align}
\frac{1}{\theta_k^2}(f(x^{(k+1)})-f^*) + \frac{R}{2}\| v^{(k+1)}-x^*\|^2 &\leq  \frac{1-\theta_k}{\theta_k^2}(f(x^{(k)})-f^*) +\frac{R}{2} \| v^{(k)}- x^*\|^2\\
&=\frac{1}{\theta_{k-1}^2} (f(x^{(k)})-f^*) +\frac{R}{2} \| v^{(k)}- x^*\|^2\\
&\leq \cdots \leq  f(x^{(1)})-f^* +\frac{R}{2} \| v^{(1)}- x^*\|^2
\end{align}
\end{subequations}
where 
the last inequality follows from $\theta_0=1$. Since $v^{(1)}=x^{(1)}$ and $f(x^{(1)})-f^*\leq \frac{R}{2}\| x^{(1)}- x^*\|^2$ from part 1) of Lemma \ref{lem01}, 
we finally obtain
\begin{equation} \label{eq:bound3}
f(x^{(k+1)})-f^* \leq R \theta_k^2 \| x^{(1)}- x^*\|^2 \le R \theta_k^2 \| x^{(1)}- x^{(1)}_\prj\|^2.
\end{equation}
Finally, we derive  $\theta_k< \frac{2}{k+2}$ for $k=0,1,2,\ldots$ from which  the sublinear convergence rate (\ref{ord1}) and its corresponding complexity will follow. From $\theta_0=1$ and Step 5 of Algorithm \ref{alg1}, we have $\theta_k >0$. From $\sqrt{\theta_k^2+4}> 2$ and Step 5 again, we have $\frac{\theta_{k+1}}{\theta_k}> \frac{2-\theta_k}{2}$ and thus $\frac{1}{\theta_{k+1}}-1=\frac{\theta_{k+1}}{\theta_k^2}> \frac{1}{\theta_k}-\frac{1}{2}= (\frac{1}{\theta_k}-1)+\frac{1}{2}$. Hence, for all $k\ge 0$, we have $\frac{1}{\theta_k}-1>\frac{k}{2}$ or $\theta_k<\frac{2}{k+2}$.
\end{proof}

\begin{algorithm}
\caption{Algorithm \ref{alg1} with restarts} \label{alg2}
\begin{tabbing}
\textbf{Input:} Initialization $y^{(0, 0)}\in \RR^n, \theta_0=1$, restart interval $K$.\\

1: \textbf{for} $j=0, 1, \cdots,$ \textbf{do}\\

2: \quad obtain $x^{(j,K)}$ by running Algorithm \ref{alg1} for $K$ iterations;\\

3: \quad set $x^{(j+1,0)}=x^{(j, K)},~ y^{(j+1,0)}=x^{(j,K)}$ and $\theta_0=1$;\\

4: \textbf{end for}
\end{tabbing}\vspace{-10pt}
\end{algorithm}

\begin{theorem}\label{thm3}
Assume that in problem \eqref{opt1}, $f\in \cR_{R,\nu}(\RR^n)$ with some $R>0, \nu >0$. Then Algorithm \ref{alg2} with $h=1/R$ and $K=\sqrt{8eR/\nu}$ reaches $\epsilon$-accuracy in $\mathcal{O}(\sqrt{\frac{R}{\nu}}\log\frac{1}{\epsilon})$ iterations.
\end{theorem}

\begin{proof}
At  iteration $j$ of Algorithm \ref{alg2}, we have
\begin{equation}
f(x^{(j+1,0)})-f^*=f(x^{(j,K)})-f^*\leq \frac{4R\cdot\|x^{(j,0)}-x_\prj^{(j,0)}\|^2_2}{K^2}\leq \frac{8R}{\nu K^2}(f(x^{j,0})-f^*)
\end{equation}
where the first inequality follows from the convergence guarantee (\ref{ord1}) of Algorithm \ref{alg1} and the second from Lemma \ref{lem2}. After $jK$ iterations, by the setting of $K=\sqrt{8eR/\nu}$ we have
\begin{equation}
f(x^{(j,0)})-f^*\leq  (\frac{8R}{\nu K^2})^j(f(x^{0,0})-f^*)= (\frac{1}{e})^{j}(f(x^{0,0})-f^*)
\end{equation}
Thus, to obtain an $\epsilon$-solution, we only need to take $j=\mathcal{O}(\log(1/\epsilon))$ and hence the total number of iterations $jK=\mathcal{O}(\sqrt{\frac{R}{\nu}}\log\frac{1}{\epsilon})$, which completes the proof.
\end{proof}
The above result and proof were motivated by \cite{n3}. Compared to \cite{n3} and  \cite{OC}, we use weaker conditions.

\section{Application to augmented $\ell_1$ minimization}\label{sc:augl1}
\subsection{An improved convergence rate}
The augmented $\ell_1$ model \eqref{augl1}
returns an exact solution to
\begin{equation}\label{BP}
\min_x \{\|x\|_1: ~Ax=b\}
\end{equation}
provided that $\alpha$ in \eqref{augl1} is large enough. For most problems  where a sparse solution $x^*$ is expected from \eqref{BP}, such as those arising in compressive sensing, paper \cite{LY} argues that $\alpha = 10\|x^*\|_\infty$ is sufficient. The Lagrange dual of \eqref{augl1}, which is problem \eqref{augl1d}, has an unconstrained and differentiable objective function. By Example 5, the negative of the dual objective function, $-f(y)$, satisfies RSC. In addition, $f$ has an $L$-Lipschitz continuous gradient $\nabla f$  with $L= \alpha \|A\|^2$. Therefore, we can apply Theorems \ref{thm:RSC+L} and \ref{thm3} to the ordinary and accelerated gradient iterations for \eqref{augl1d}.

The gradient ascent iteration for  \eqref{augl1d} is known as the linearized Bregman algorithm (LBreg):
\begin{subequations}\label{form1}
\begin{align}\label{iter1}
&x^{(k+1)}\leftarrow\alpha \shrink (A^Ty^{(k)}),\\
&y^{(k+1)}\leftarrow y^{(k)} +h (b-Ax^{(k+1)}),\label{iter1b}
\end{align}
\end{subequations}
where $x^{(k)}$ and $y^{(k)}$ are the primal and dual variables at iteration $k$ and $h>0$ is the step size. One can verify that $(b-Ax^{(k+1)})$ is the gradient to the objective of \eqref{augl1d}. The solution set is given by
\begin{equation}\label{set1}
\mathcal{Y}^*=\{y\in\RR^m: b-\alpha A \shrink (A^Ty)=0\}=\{y\in\RR^m: \alpha \shrink (A^T{y})=x^*\}
\end{equation}
where  $x^*$ is assumed to be the unique solution to \eqref{augl1}; the derivation can be found in \cite{LY}.

Paper \cite{LY} shows
$$\|y^{(k)}-y^{(k)}_\prj\|\leq \sqrt{1-\left(\frac{\nu}{L}\right)^2}\;\|y^{(k-1)}-y^{(k-1)}_\prj\|.$$
Applying Theorem \ref{thm:RSC+L}, we obtain a tighter convergence bound:
\begin{theorem}\label{thm4}
In problem (\ref{augl1d}), assume that $A\in\RR^{m\times n}$ and $b\in\RR^m$ are nonzero and $Ax=b$ are consistent. Let $f^*$ be the optimal objective value of (\ref{augl1d}). The linearized Bregman iteration (\ref{form1}) starting from any $y^{(0)}\in \RR^n$ with step size $h_k=\frac{1}{L}$ generates a  Q-linearly converging sequence $\{y^{(k)}\}$
\begin{equation}\label{resu01}
\|y^{(k)}-y^{(k)}_\prj\|\leq \sqrt{1-\frac{\nu}{L}}\|y^{(k-1)}-y^{(k-1)}_\prj\|, \quad \forall k\ge 1.
\end{equation}
 The objective value converges R-linearly as
\begin{equation}\label{resu02}
f^*-f(y^{(k)})\leq \frac{L}{2}\|y^{(0)}-y^{(0)}_\prj\|^2(1-\frac{\nu}{L})^{k}, \quad \forall k\ge 1.
\end{equation}
Furthermore, $x^{(k)}$ converges R-linearly as
\begin{equation}\label{resu03}
\|x^{(k+1)}-x^*\| \leq L\|y^{(0)}-y^{(0)}_\prj\|(1-\frac{\nu}{L})^{k/2},\quad \forall k\ge 1,
\end{equation}
where $x^*$ is the solution to (\ref{augl1}). The results are in the global sense.
\end{theorem}

\begin{proof}
Due to (\ref{iter1}), (\ref{set1}), the expression $\nabla f(y)=b -\alpha \shrink(A^Ty)$,  and the Lipschitz property (\ref{Lip}) of $\nabla f(y)$, we have
\begin{subequations}
\begin{align}
\|x^{(k+1)}-x^* \| &= \|\alpha \shrink (A^Ty^{(k)})-\alpha \shrink (A^T{y^{(k)}_\prj})\|,\\
 &= \| \nabla f(y^{(k)})-\nabla f({y^{(k)}_\prj})\|,\\
 &\leq L \|y^{(k)}-{y^{(k)}_{prj}}\|.
\end{align}
\end{subequations}
which gives (\ref{resu03}). The remained results follow from Theorem \ref{thm:RSC+L} applied to $-f$.
\end{proof}

\subsection{Numerical simulation}
To demonstrate the convergence results, we compared the following algorithms for problem \eqref{augl1d}:
\begin{enumerate}
\item fixed-step gradient ascent (Algorithm \ref{alg0});
\item gradient ascent with Nesterov's acceleration (Algorithm \ref{alg1}, \cite{HMG});
\item Nesterov's acceleration with \emph{restart} (Algorithm \ref{alg3} with \emph{restart});
\item Nesterov's acceleration with \emph{skip} (Algorithm \ref{alg3} with \emph{skip}).
\end{enumerate}
\begin{algorithm}
\caption{Nesterov's accelerated gradient method with \emph{reset}} \label{alg3}
\begin{tabbing}
\textbf{Input:} Initialization $y^{(0)}\in \RR^n, \theta_0=1$, and $h>0$.\\

1: \textbf{for} $k=0, 1, \cdots,$ \textbf{do}\\

2: \quad $x^{(k+1)}=y^{(k)}-h\nabla f(y^{(k)})$; \quad (negative gradient step)\\

3: \quad If \emph{restart} then\\

4: \quad\quad $\theta_k=1$ and $\beta_{k+1}=0$;\\

5: \quad elseif \emph{skip} then\\

6: \quad\quad $\beta_{k+1}=0$;\\

7: \quad else\\

8: \quad\quad $\beta_{k+1}=(1-\theta_k)(\sqrt{\theta_k^2+4}-\theta_k)/2$; \quad (extrapolation weight)\\

9: \quad End if \\

10: \quad $y^{(k+1)}=x^{(k+1)}+\beta_{k+1}(x^{(k+1)}-x^{(k)})$; \quad (extrapolation)\\

11: \quad $\theta_{k+1}=\theta_k(\sqrt{\theta_k^2+4}-\theta_k)/2$; \quad (dampening of acceleration parameter)\\

6: \textbf{end for}
\end{tabbing}\vspace{-10pt}
\end{algorithm}
Although for  \eqref{augl1d} we can compute $K=\sqrt{8eL/\nu}$ using the lower bound of $\nu$ given in Example 5 and thus run Algorithm \ref{alg2} with restart every $K$ iterations, such $K$ was found too large. Instead, we ran Algorithm \ref{alg3}, which uses the following scheme to trigger \emph{restart}  as suggested in \cite{OC} (the inequality is given in the opposite directions for  concave maximization):
$$\textrm{Gradient scheme: }\nabla f(y^{(k-1)})^T(y^{(k)}-y^{(k-1)})<0.$$
We also introduce the \emph{skip} heuristic:  set $\beta_{k+1}=0$ (and make \emph{no} change to $\theta_k$).

The comparisons use two examples. Each had sparse signals $x^o$ with 512 entries, out of which 25 were nonzero entries sampled independently from the standard Gaussian distribution (Test 1, Figure \ref{fig3}(a)) or set to $\pm 1$ uniformly randomly (Test 2, Figure \ref{fig3}(b)). Both examples have the same sensing matrix $A$ with 256 rows and entries sampled independently from the standard Gaussian distribution.   We used the following parameters:  $b=Ax^o$, $\alpha=10\|x^o\|_\infty$, and $h=\frac{1}{L}=\frac{1}{\alpha \|A\|^2}$. All iterations were stopped upon $\|Ax^{(k)}-b\|<10^{-14}\|b\|$. Figure \ref{fig3} depicts the relative error  $\frac{\|x^{(k)}-x^0\|}{\|x^0\|}$ versus iteration $k$.
\begin{figure}[ht]
\centering
\subfigure[Test 1: Gaussian sparse vector recovery]{
    \includegraphics[width=0.45\textwidth]{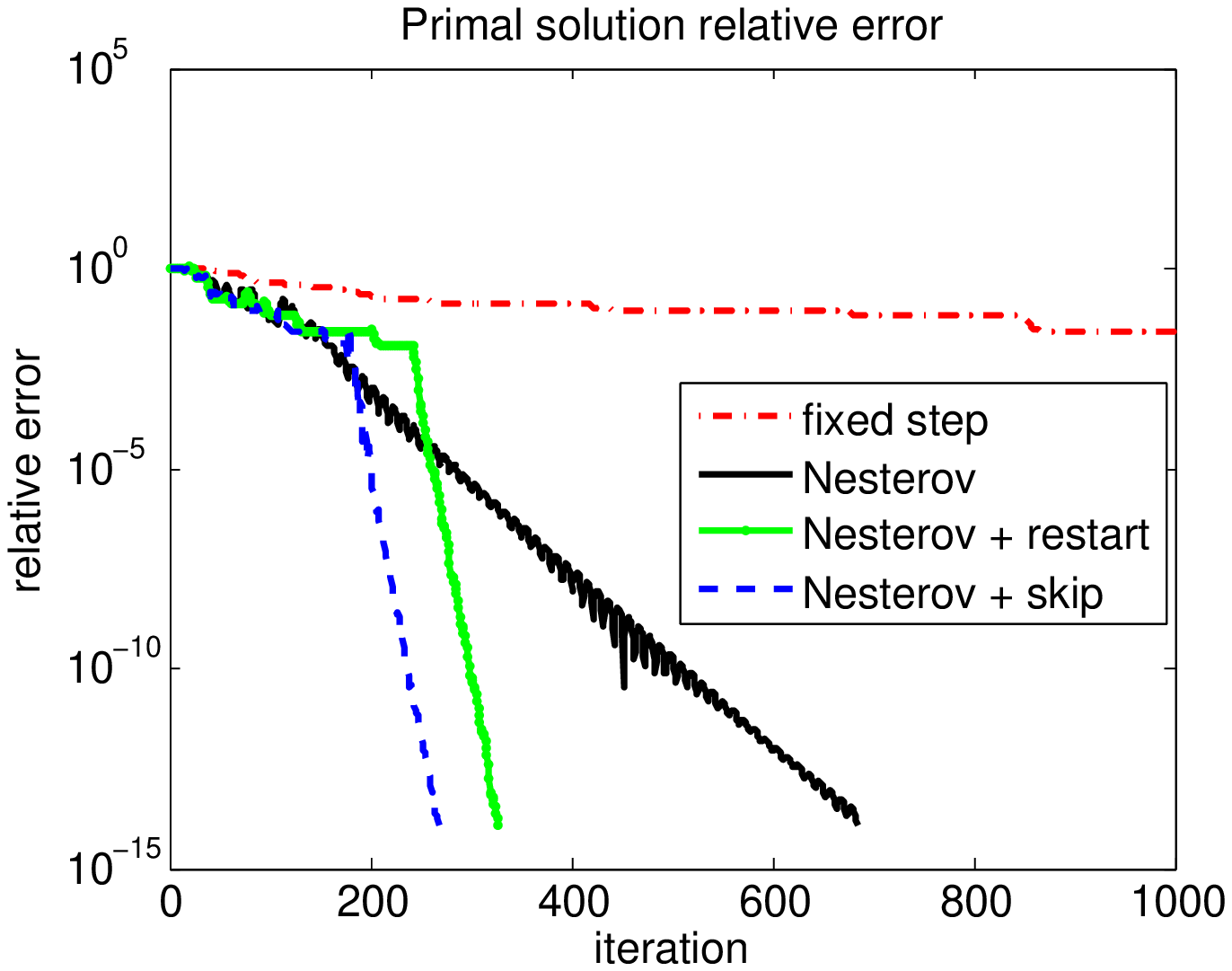}
    \label{fig:subfig1}
}
\subfigure[Test 2: Bernoulli sparse vector recovery]{
   \includegraphics[width=0.45\textwidth]{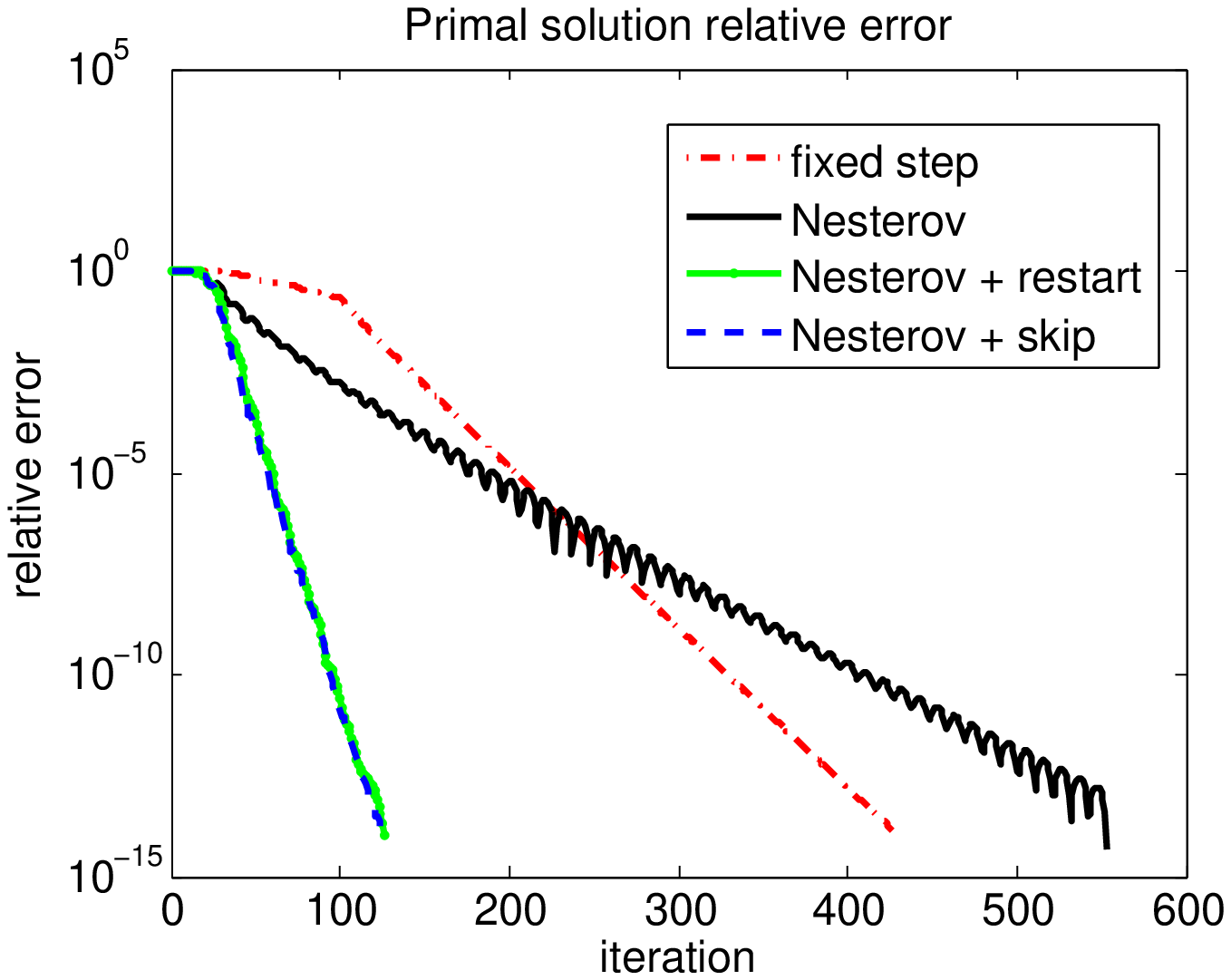}
    \label{fig:subfig2}
}
\caption{Relative error of primal variable $x^{(k)}$}
\label{fig3}
\end{figure}

The fixed-step gradient iteration converged very slowly in Test 1, much slower than in Test 2; this can be explained by a smaller $\nu$ in Test 1 (see Lemma 7 of  \cite{LY} for an explicit lower bound of $\nu$). The fixed-step iteration exhibited a linear-convergence behavior in Test 2 though we  cannot tell the same from Test 1,.

The accelerated gradient method performed similarly in both tests. Its performance was significantly improved in the second phase by \emph{restart} and \emph{skip}. In Test 1, \emph{skip} was more effective. The two schemes did not appear to make much difference in these tests. It is interesting to note that in Test 2, both \emph{restart} and \emph{skip} had  faster  rates of convergence than the fixed-step gradient iteration; this deserves further tests and perhaps theoretical investigation.

As the focus of this paper is not numerical simulation, we do not present more numerical results. For the interested reader, the source code can be found on the second author's homepage.

\section{Conclusions}\label{sc:concl}
The  convergence behavior of gradient methods on convex differentiable functions is one of the core questions in convex optimization. It is  known to many researchers that global Lipschitz continuity of $\nabla f$ is more than sufficient for sublinear convergence and asking $f$ to be strongly convex is also too much for linear convergence. For the ordinary and accelerated gradient methods, this paper shows using rather straightforward steps that these conditions restricted to certain line segments are sufficient for the existing convergence results to hold. In addition, it shows that strong convexity restricted to between current point $x$ and its projection to the solution set is also necessary for the geometric decay of solution error.

For the accelerated gradient method to achieve the best worst-case bound $O\left(\sqrt{\frac{R}{\nu}}\log\frac{1}{\epsilon}\right)$ on (restricted) strongly convex functions, the  modulus $\nu$ of the objective function must be given. This is not practical. It is an open question to design  a method with this bound but not requiring the knowledge of $\nu$. On the other hand, the \emph{restart} and \emph{skip} \emph{heuristics} appear to improve the performance of the accelerated method.

\section*{Acknowledgements}

We want to thank Profs. Q. Ling, S. Ma, Z. Wen, and Y. Zhang and graduate students Z. Peng, Y. Xu, and T. Sun for discussions and corrections. The work of H. Zhang is supported by China Scholarship Council during his visit to Rice University, and in part by Graduate School of NUDT under Funding of Innovation B110202, Hunan Provincial Innovation Foundation For Postgraduate CX2011B008, and National Science Foundation of China under Grants No. 61271014 and No.61072118. H. Zhang thanks Rice University, CAAM Department, for hosting him. The work of W. Yin is supported in part by NSF grants DMS-0748839 and ECCS-1028790, and ONR Grant N00014-08-1-1101.

\section*{Appendix}
We select the parameter $\theta$ and step size $h$ in (\ref{ineq2}) to minimize the upper bound. Let $r=\frac{\nu}{2R}, h>0$. As we need to deal with the second term in (\ref{ineq2}), two cases are studied below depending on the sign of $h^2-\frac{\theta h}{R}$:

\textbf{Case A:} $h^2-\frac{\theta h}{R}\leq 0$, i.e., $h\in (0, \frac{\theta}{R}], \theta\in[0,1]$. Applying the Cauchy-Schwartz inequality to RSI, we get
\begin{equation}
\|\nabla f(x^{(k)})-\nabla f(x_\prj^{(k)})\|^2\geq \nu^2 \|x^{(k)}-x_\prj^{(k)}\|^2.
\end{equation}
From $h^2-\frac{\theta h}{R}\leq 0$ and (\ref{ineq2}), we derive that
\begin{subequations}
\begin{align}
\|x^{(k+1)}-x_\prj^{(k+1)} \|^2 &\leq  (1-2(1-\theta)\nu h)\|x^{(k)}-x_\prj^{(k)}\|^2+\nu^2(h^2-\frac{\theta h}{R}) \|x^{(k)}-x_\prj^{(k)}\|^2,\\
 &= \left(\nu^2h^2 -2\left((1-\theta)\nu+\frac{\theta\nu^2}{2R}\right) h+1\right)\|x^{(k)}-x_\prj^{(k)}\|^2,\\
 \label{func1}
 &\triangleq f_1(\theta, h)\|x^{(k)}-x_\prj^{(k)}\|^2.
\end{align}
\end{subequations}
Let $h_0=\frac{\theta}{2R}+\frac{(1-\theta)}{\nu}$, which is the minimum point of the quadratic function $f_1(\theta, h)$ over variable $h$ for each fixed $\theta$. To determine whether such $h_0$ is included in the interval $(0, \frac{\theta}{R}]$, we consider $h_0=\frac{\theta}{2R}+\frac{(1-\theta)}{\nu}=\frac{\theta}{R}$ and get $\theta= \frac{1}{1+r}$. Now, we split the interval $[0,1]$ into $[\frac{1}{1+r}, 1]$ and $[0, \frac{1}{1+r})$. If $\theta\in [\frac{1}{1+r}, 1]$, we have $\frac{\theta}{R}\geq h_0$ which means the point $h_0\in (0, \frac{\theta}{R}]$. Thus,
$$\min_{h\leq \frac{\theta}{R}, \frac{1}{1+r}\leq \theta\leq 1}  f_1(\theta, h)=\min_{ \frac{1}{1+r}\leq \theta\leq 1}  f_1(\theta, h_0)=\min_{ \frac{1}{1+r}\leq \theta\leq 1}  1-(1-(1+r)\theta)^2= 1-r^2,$$
where the minimum value $1-r^2$ is obtained at $\theta=1$ and $h=h_0=\frac{1}{2R}$.  If $\theta\in [0, \frac{1}{1+r})$,  we have $\frac{\theta}{R}< h_0$ which means the point $h_0\notin (0, \frac{\theta}{R}]$. By monotone decreasing of $f_1(\theta, h)$ on the interval $h\leq \frac{\theta}{R}$  for each fixed $\theta$, we have
$$\min_{h\leq \frac{\theta}{R},0\leq \theta< \frac{1}{1+r}}  f_1(\theta, h)=\min_{ 0\leq \theta< \frac{1}{1+r}}  f_1(\theta, \frac{\theta}{R})=\min_{ 0\leq \theta< \frac{1}{1+r}}  1-4\theta(1-\theta)r= 1-r$$
 where the minimum value $1-r$ is obtained at $\theta=\frac{1}{2}$  and $h=\frac{\theta}{R}=\frac{1}{2R}$; note that $\frac{1}{2}\in [0, \frac{1}{1+r})$ since $r<1$. Therefore, on the intervals $h\in (0, \frac{\theta}{R}]$ and $\theta\in[0,1]$, the minimum value $1-r$ of $f_1(\theta, h)$ is obtained at $(\theta, h)=(\frac{1}{2},\frac{1}{2R})$.

\textbf{Case B:} $h^2-\frac{\theta h}{R}\geq 0$, i.e., $h\in [\frac{\theta}{R}, +\infty), \theta\in[0,1]$. Applying the Cauchy-Schwartz inequality to part 2) of Lemma \ref{lem01}, we get
\begin{equation}
\|\nabla f(x^{(k)})-\nabla f(x_\prj^{(k)})\|^2\leq 4R^2 \|x^{(k)}-x_\prj^{(k)}\|^2.
\end{equation}
From $h^2-\frac{\theta h}{R}\geq 0$ and (\ref{ineq2}), we derive that
\begin{subequations}
\begin{align}
\|x^{(k+1)}-x_\prj^{(k+1)} \|^2 &\leq  (1-2(1-\theta)\nu h)\|x^{(k)}-x_\prj^{(k)}\|^2+4R^2(h^2-\frac{\theta h}{R}) \|x^{(k)}-x_\prj^{(k)}\|^2,\\
 &= (4R^2h^2 -2(2\theta R+(1-\theta)\nu) h+1)\|x^{(k)}-x_\prj^{(k)}\|^2,\\
 &\triangleq f_2(\theta, h)\|x^{(k)}-x_\prj^{(k)}\|^2.
\end{align}
\end{subequations}
Let $h_1=\frac{2\theta R+(1-\theta)\nu}{4R^2}$, which is the minimum point of the quadratic function $f_2(\theta, h)$ over variable $h$ for each fixed $\theta$. Similarly, we split the interval $[0,1]$ into $(\frac{r}{1+r}, 1]$ and $[0, \frac{r}{1+r}]$. If $\theta\in(\frac{r}{1+r}, 1]$, we have $\frac{\theta}{R}>  h_1$ which means $h_1\notin [\frac{\theta}{R}, +\infty)$. By monotone increasing of $f_2(\theta, h)$ on the interval $h\geq \frac{\theta}{R}$  for each fixed $\theta$, we have
$$\min_{h\geq \frac{\theta}{R}, \frac{r}{1+r} <\theta\leq 1}  f_2(\theta, h)=\min_{  \frac{r}{1+r} <\theta\leq 1}  f_2(\theta, \frac{\theta}{R})=\min_{  \frac{r}{1+r} <\theta\leq 1}  1-4\theta(1-\theta)r = 1-r,$$
where the minimum value $1-r$ is obtained at $\theta=1/2$ and $h=\frac{\theta}{R}=\frac{1}{2R}$; note that $\frac{1}{2}\in (\frac{r}{1+r}, 1]$ since $r<1$.  If $\theta\in [0, \frac{r}{1+r}]$, we have $\frac{\theta}{R}\leq  h_1$ which means $h_1\in [\frac{\theta}{R}, +\infty)$. Thus,
$$\min_{h\geq \frac{\theta}{R},0\leq \theta\leq \frac{r}{1+r}}  f_2(\theta, h)=\min_{ 0\leq \theta\leq  \frac{r}{1+r}}  f_2(\theta, h_1)=\min_{ 0\leq \theta\leq  \frac{r}{1+r}}  1-(\frac{2\theta R+(1-\theta)\nu}{2R})^2=1- (\frac{2\nu}{2R+\nu})^2,$$ where the minimum value is obtained at $\theta=\frac{r}{1+r}$ and $h=h_1$. After simple calculations, it holds $r=\frac{\nu}{2R}>(\frac{2\nu}{2R+\nu})^2$ and hence $1-r< 1-(\frac{2\nu}{2R+\nu})^2$.
Therefore, on the intervals $h\in [\frac{\theta}{R}, +\infty)$ and $\theta\in[0,1]$, the minimum value $1-r$ of $f_2(\theta, h)$ is obtained at $(\theta, h)=(\frac{1}{2},\frac{1}{2R})$ as well.

\small{

}

\end{document}